\documentclass{amsart}
\usepackage{amsfonts}
\usepackage{amsmath}
\usepackage{amssymb}
\usepackage{graphicx}
\newtheorem{theorem}{Theorem}
\theoremstyle{plain}

\newtheorem{conjecture}{Conjecture}
\newtheorem{corollary}{Corollary}

\newtheorem{definition}{Definition}
\newtheorem{example}{Example}

\newtheorem{lemma}{Lemma}

\newtheorem{problem}{Problem}

\newtheorem{remark}{Remark}

\numberwithin{equation}{section}

\newcommand{\seq}[1]{\{#1\}_{k=0}^{\infty}}
\newcommand{\R}{\mathbb{R}}

\title[Monotone linear operators]{Sufficient conditions for a linear operator on $\R[x]$ to be monotone}

\author{Leah Buck${}^{\dag}$, Kelly Emmrich${}^{\ddag}$ and Tam\'as Forg\'acs${}^{\star}$}
\thanks{Research partially completed during the 2016 Fresno State Mathematics REU, supported by NSF grant DMS-1460151. }

\begin{document}

\maketitle
\begin{abstract} We demonstrate that being a hyperbolicity preserver does not imply monotonicity for infinite order differential operators on $\mathbb{R}[x]$, thereby settling a recent conjecture in the negative. We also give some sufficient conditions for such operators to be monotone.\\
{\bf MSC 30C15, 26C10}
\end{abstract}
\section{Introduction}
\smallbreak
Any linear operator $T:\mathbb{R}[x] \to \mathbb{R}[x]$ may be represented uniquely as a formal series in powers of $D:=\frac{d}{dx}$,
\begin{equation}\label{diffoprep}
T=\sum_{k=0}^{\infty} Q_k(x)D^k.
\end{equation}
Such operators on $\mathbb{R}[x]$ and their properties have been investigated by several researchers in the recent years (see \cite{bates}, \cite{BB}, \cite{bo}, \cite{chasse}, \cite{chassegrabarekvisontai}, \cite{tominvolve}, \cite{tomandrzej1}, \cite{tomandrzej2}, and \cite{yoshi} for a few examples, and the extensive works of T. Craven and G. Csordas). The work of many of these researchers was motivated by the P\'olya-Schur program, and focused on determining when a given linear operator $T$ on $\mathbb{R}[x]$ has the property that $T[p(x)]$ has only real zeros whenever $p(x)$ has only real zeros. We call such operators {\it hyperbolicity preservers}. In \cite{BB}  J. Borcea and P. Br\"and\'en gave a complete characterization of these operators. Following a few definitions, we recall their theorem for the reader's convenience, as we will make use of it both in constructing some examples, as well as in our main result, Theorem \ref{thm:counter}.
\begin{definition} $\mathcal{H}_1(\mathbb{R})$ denotes the set of all real hyperbolic univariate polynomials. $\overline{\mathcal{H}}_2(\mathbb{R})$ denotes the closure under locally uniform convergence of the set of all real stable bivariate polynomials, i.e. polynomials $p(z,w)$ which do not vanish if $\Im z>0$ and $\Im w >0$. 
\end{definition} 
\begin{definition} \label{def:symbol} The symbol $G_T(z,w)$ of a liner operator $T:\mathbb{R}[x] \to \mathbb{R}[x]$ is defined as the formal power series
\[
G_T(z,w)=\sum_{k=0}^{\infty} \frac{(-1)^kT[z^k]w^k}{k!}.
\]
\end{definition}
\begin{theorem}{\cite[Theorem 5]{BB}} \label{bb} A linear operator $T:\mathbb{R}[x] \to \mathbb{R}[x]$ preserves hyperbolicity if, and only if, either
\begin{itemize}
\item[(a)] $T$ has range of dimension at most two and is of the form 
\[
T(f)=\alpha(f)P+\beta(f)R, 	\qquad f \in \mathbb{R}[x],
\]
where $\alpha, \beta:\mathbb{R}[x]\to\mathbb{R}$ are linear functionals, and $P,R \in \mathcal{H}_1(\mathbb{R})$ have interlacing zeros, or
\item[(b)] $G_T(z,w) \in \overline{\mathcal{H}}_2(\mathbb{R})$, or 
\item[(c)] $G_T(z,-w) \in \overline{\mathcal{H}}_2(\mathbb{R})$.
\end{itemize} 
\end{theorem}
Some of the above cited works studied whether a linear operator $T:\mathbb{R}[x] \to \mathbb{R}[x]$ is a hyperbolicity preserver, under the added hypothesis that $T$ be diagonal with respect to one of the classical orthogonal bases. As such, these works {\it by definition} investigated what we now call Hermite- , Laguerre\footnote{The term `Laguerre multiplier sequence' was originally used by T. Craven amd G. Csordas to designate a special kind of classical multiplier sequence. The reader should be aware of the dual use of the terminology.}- , Legendre- and Jacobi multiplier sequences. \\
\indent  Given the representation in (\ref{diffoprep}), it is natural to ask what properties of the coefficient polynomials $Q_k(x)$ encode properties of the operator they represent. In \cite{tomandrzej1}, the third author and A. Piotrowski connect the hyperbolicity of the $Q_k(x)$s to the operator $T$ being a hyperbolicity preserver (albeit again in a setting that assumes diagonality). In the present paper we seek to understand how another property of the coefficient polynomials $Q_k(x)$ is related to hyperbolicity preservation. To this end, we make the following definitions.
\begin{definition} Let $\displaystyle{T=\sum_{k=0}^{\infty} Q_k(x)D^k}$ be a linear operator on $\mathbb{R}[x]$. If $\deg Q_k \leq \deg Q_{k+1}$ for all $k$ such that $\deg Q_k \geq 0$\footnote{We adopt the convention that the degree of the identically zero polynomial is $-\infty$.}, we call the operator $T$ monotone. 
\end{definition}
\begin{definition} \label{def:order} Let $\displaystyle{T=\sum_{k=0}^{\infty} Q_k(x)D^k}$ be a linear operator on $\mathbb{R}[x]$. 
\begin{itemize}
\item[(i)] If there exists an $N$, such that $Q_k(x)\equiv 0$ for all $k \geq N$, we say that $T$ is a finite order differential operator. 
\item[(ii)] If $Q_k(x) \not \equiv 0$ for infinitely many $k$, we say that $T$ is an infinite order differential operator. 
\end{itemize}
\end{definition}
While studying Legendre multiplier sequences in \cite{tominvolve}, the third author, J. Haley, R. Menke and C. Simon  were led to the following conjecture regarding the monotonicity of a hyperbolicity preserving infinite order differential operator.
\begin{conjecture}\cite[Conjecture 19,\,p.785]{tominvolve} \label{conj:tom} Suppose that $\displaystyle{T=\sum_{k=0}^{\infty}T_k(x)D^k}$ is an infinite order differential operator. If $T$ is not monotone, then $T$ is not hyperbolicity preserving.
\end{conjecture}
In order to further motivate the conjecture, we note that it is straightforward to construct infinite order, monotone differential operators, if hyperbolicty preservation is not required at the same time. Similarly, if one is allowed to consider finite order operators, then monotonicity does not follow from hyperbolicity preservation as witnessed for example by the operator $T=xD+D^2$. Indeed, a quick calculation shows that 
\[
G_T(z,-w)=(z+w)we^{zw} \in \overline{\mathcal{H}}_2(\mathbb{R}),
\]
whence $T$ is hyperbolicity preserving, but clearly not monotone. Thus, considering infinite order differential operators, which are simultaneously also hyperbolicity preservers, is natural when one is looking for operators that are necessarily monotone.\\
\indent Alas, the next section of the current paper demonstrates that Conjecture \ref{conj:tom} is in fact false. Section \ref{s:3} contains several sufficient conditions for a linear operator $T$ on $\mathbb{R}[x]$ to be monotone, illuminating connections between the operator properties of diagonality (cf. Definition \ref{def:diagonal}), hyperbolicity preservation, order, and monotonicity. The paper concludes with a section on open problems. 
\section{An infinite order hyperbolicity preserver that is not monotone} \label{s:2}
We now construct a hyperbolicity preserving infinite order (cf. Definition \ref{def:order}) differential operator that is not monotone, using condition (a) in Theorem \ref{bb}. 
\begin{lemma}\label{lem:notmonotone}  Let $\alpha, \beta: \mathbb{R}[x] \to \mathbb{R}$ be linear functionals defined on the standard basis by 
\begin{eqnarray*}
&& \alpha(1)=0, \quad \alpha(x)=1, \quad \alpha(x^n)=0 \qquad \forall n\geq 2, \qquad \text{and} \\
&& \beta(1)=1, \quad \beta(x)=0, \quad \beta(x^n)=0 \qquad \forall n \geq 2.
\end{eqnarray*}
Let $P(x)=x(x+1)(x-1)$, and $\displaystyle{R(x)=x^2-\frac{1}{4}}$. The linear operator
\[
T[f]=\alpha(f)P(x)+\beta(f) R(x), \qquad (f \in \mathbb{R}[x])
\]
is not monotone. 
\end{lemma}
\begin{proof} The coefficient polynomials in the representation (\ref{diffoprep}) can be calculated recursively (see for example \cite[Proposition 29, p.\,32]{andrzej}) as
\begin{equation}\label{eq:diffoprep}
Q_n(x)=\frac{1}{n!} \left(T[x^n]-\sum_{k=0}^{n-1} Q_k(x)D^k[x^n]\right), \qquad (n=0,1,2,\ldots).
\end{equation}
Using the definition of $T$ as in the statement of the lemma we compute
\begin{eqnarray*}
Q_0(x)&=&T[1]=0 \cdot P(x)+1 \cdot R(x)=x^2-\frac{1}{4}, \quad\text{and}\\
Q_1(x)&=& T[x]-Q_0(x)x=P(x)-x\left(x^2-\frac{1}{4}\right)=-\frac{3}{4}x.
\end{eqnarray*} 
It follows that $\deg Q_0(x)=2>1=\deg Q_1(x)$, and hence $T$ is not monotone.
\end{proof}
We next show that the operator under consideration is hyperbolicity preserving.
\begin{lemma} The operator $T$ defined in Lemma \ref{lem:notmonotone} is a hyperbolicity preserver.
\end{lemma}
\begin{proof} By Theorem \ref{bb}, if condition (a) holds, the operator is a hyperbolicity preserver. By construction, $T$ has the required form. In addition, $P(x)$ and $R(x)$ are hyperbolic polynomials with interlacing zeros. The result follows.
\end{proof}
In order to complete our discussion, we need to verify that $T$ is indeed an infinite order differential operator. 
\begin{lemma} The operator $T$ as defined in Lemma \ref{lem:notmonotone} is an infinite order differential operator. 
\end{lemma}
\begin{proof} We establish the result by showing that none of the coefficient polynomials $Q_k(x),\, k \geq 2$ vanishes. In fact, we claim that
\begin{equation} \label{eq:qform}
Q_k(x)=\frac{(-1)^{k+1}(k-1)}{k!}Q_0(x)x^k+\frac{(-1)^{k+1}}{(k-1)!}x^{k-1}Q_1(x), \qquad k \geq 2.
\end{equation}
To see this, we calculate $Q_2(x)$ and $Q_3(x)$ directly:
\begin{eqnarray*}
Q_2(x)&=&-\frac{1}{2}Q_0(x)x^2-xQ_1(x),\\
Q_3(x)&=&\frac{1}{3}Q_0(x)x^3+\frac{1}{2}x^2Q_1(x).
\end{eqnarray*}
Assume now that $Q_k(x)$ is of the form (\ref{eq:qform}) for $2 \leq k \leq n$. Using the recurrence relation (\ref{eq:diffoprep}) we obtain
\begin{eqnarray*}
&&Q_{n+1}(x)\stackrel{(1)}{=}\frac{1}{(n+1)!}\left(-Q_0(x)x^{n+1}-(n+1)Q_1(x)x^n-\sum_{k=2}^n Q_k(x)D^k[x^{n+1}] \right) \\
&=&\frac{1}{(n+1)!} \left(-Q_0(x)x^{n+1}-(n+1)Q_1(x)x^n-\right.\\&-&\left.\sum_{k=2}^n \left[\frac{(-1)^{k+1}(k-1)}{k!}Q_0(x)x^k+\frac{(-1)^{k+1}}{(k-1)!}x^{k-1}Q_1(x)\right] \frac{(n+1)!}{(n+1-k)!}x^{n+1-k} \right)\\
&=&\frac{1}{(n+1)!}\left\{\left(-1+\sum_{k=2}^n \frac{(-1)^k(k-1) (n+1)!}{k!(n+1-k)!} \right)Q_0(x)x^{n+1}\right.\\
&+&\left.(n+1)\left(\sum_{k=0}^{n-1} \binom{n}{k}(-1)^{k+1} \right)x^{n}Q_1(x)\right\}\\
&\stackrel{(4)}{=}&\frac{(-1)^{n+2}n}{(n+1)!}Q_0(x)x^{n+1}+\frac{(-1)^{n+2}}{n!}x^{n}Q_1(x),
\end{eqnarray*}
where equality (1) is a consequence of $T[x^n] \equiv 0$ for all $n \geq 2$, and equality (4) uses the binomial identity
\[
0=\sum_{k=0}^n \binom{n}{k}(-1)^k
\]
in the simplification of both sums. The proof is complete. 
\end{proof}
The preceding three lemmas constitute the proof of the following theorem.
\begin{theorem} \label{thm:counter} There exist infinite order hyperbolicity preservers which are not monotone.
\end{theorem}
\begin{remark} We point out that the functionals $\alpha$ and $\beta$ could have been defined in many different ways. For example, the values of $\alpha(x^n)$ and $\beta(x^n)$ could be recursively picked for $n \geq 2$ in infinitely many different ways, while ensuring that the $Q_k(x)$s do not vanish for any $k=0,1,2,\ldots$. Instead of an existential proof, however, we opted for a constructive one.
\end{remark}
\section{Sufficient conditions for monotonicity} \label{s:3}
Now that we know that monotonicity of a linear operator $T$ on $\mathbb{R}[x]$ is not equivalent to $T$ being a hyperbolicity preserver, we provide some sufficient conditions for such an operator to be monotone. In certain cases, diagonality alone will suffice (cf. Theorem \ref{thm:classicaldiag} and Corollary \ref{cor:classicaldiag}, as well as Theorem \ref{thm:affinetransform}). In other instances, diagonality and hyperbolicty preservation will be required (cf. Theorem \ref{thm:hermitediag}). First, however, we show that one extra condition on the functionals as in part (a) of Theorem \ref{bb} will guarantee that the operator they define is in fact monotone. 
\begin{theorem} \label{thm:monotoneranktwo}
 Suppose that $T$ is a linear operator as defined in part (a) of Theorem \ref{bb}. If $\alpha(1)\neq 0$ and $\beta(1)\neq 0$, then as a differential operator $T$ is of infinite order, with 
 \begin{equation} \label{eq:monotonefunctional}
 Q_k(x)=\frac{(-1)^k}{k!} Q_0(x)\cdot x^k+\text{lower order terms} \qquad \text{for all} \quad k=0,1,2,\ldots.
 \end{equation}
 Consequently, $T$ is monotone. 
\end{theorem}
\begin{proof} Suppose $\alpha(0) \neq 0$ and $\beta(0) \neq 0$, and let $P,R$ be two polynomials forming a basis for the range of $T$ (cf. (a) of Theorem \ref{bb}). Then 
\begin{eqnarray*}
Q_0(x)&=&\frac{(-1)^0}{0!}Q_0(x), \quad \text{and}\\
Q_1(x)&=&\alpha(1)P+\beta(1)R-Q_0(x)\cdot x\\
&=&\frac{(-1)^1}{1!}Q_0(x) \cdot x+\text{lower order terms.}
\end{eqnarray*}
Suppose now that (\ref{eq:monotonefunctional}) holds for $0 \leq k \leq n-1$. Equation (\ref{eq:diffoprep}) then implies that  
\begin{eqnarray*}
Q_n(x)&=&\frac{1}{n!}\left(T[x^n]-\sum_{k=0}^{n-1}\left\{ \frac{(-1)^k}{k!}Q_0(x)\cdot x^k+\text{lower order terms}\right\} \frac{n!}{(n-k)!} x^{n-k} \right)\\
&=&\frac{1}{n!}\left(-Q_0(x)\cdot x^n\sum_{k=0}^{n-1} \frac{(-1)^k n!}{k!(n-k)!}\right)+\text{lower order terms}\\
&=&\frac{(-1)^n}{n!}Q_0(x) \cdot x^n +\text{lower order terms}.
\end{eqnarray*}
The monotonicity of $T$ now readily follows, since $\deg Q_k(x)=\deg Q_0(x)+k$ for all $k=0,1,2,\ldots$. 
\end{proof}
Before we proceed with our next result, we need to make the following definition.
\begin{definition} \label{def:diagonal} A linear operator $T:\mathbb{R}[x] \to \mathbb{R}[x]$ is called diagonal if there exists a basis $\seq{B_k(x)}$, and a sequence of real numbers $\seq{\gamma_k}$ called the eigenvalues of $T$, such that 
\[
T[B_k(x)]=\gamma_k B_k(x), \qquad k=0,1,2,\ldots.
\]
\end{definition}
We now show that if $T$ is a linear operator that is diagonal with respect to an affine transformation of the standard basis, then $T$ is a monotone operator. This result is an extension of the following theorem of Piotrowski.
\begin{theorem}\cite[Proposition 33]{andrzej} \label{thm:classicaldiag}
Suppose that $T:\mathbb{R}[x] \to \mathbb{R}[x]$ is a linear operator that is diagonal with respect to the standard basis, with eigenvalues $\seq{\gamma_k}$. Then 
\begin{equation}\label{eq:classicaldiag}
T=\sum_{k=0}^{\infty} \frac{g_k^*(-1)}{k!}x^k D^k,
\end{equation}
where $\displaystyle{g_k^*(x)=\sum_{j=0}^k \binom{k}{j}x^{k-j}}$ are the reversed Jensen polynomials  associated to the sequence $\seq{\gamma_k}$.
\end{theorem} 
\begin{corollary}\label{cor:classicaldiag} Suppose that $T:\mathbb{R}[x] \to \mathbb{R}[x]$ is a linear operator that is diagonal with respect to the standard basis. Then $T$ is monotone. 
\end{corollary}
\begin{proof} Given the representation in (\ref{eq:classicaldiag}), we see that $Q_k(x)=g_k^*(-1)x^k$. It is thus immediate that for every $k \geq 0$, either $Q_k(x) \equiv 0$, or $\deg Q_k(x)=k$. The result follows.
\end{proof}
\begin{theorem} \label{thm:affinetransform} Consider any affine transformation of the standard basis given by $q_k(x)=c_k(\alpha x + \beta)^k,$ where  $\{c_k\}_{k=0}^\infty$ is a sequence of non-zero real numbers, and $ \alpha,\beta\in\mathbb{R},~\alpha \neq 0.$ Suppose the linear operator $T$  is diagonal with respect to $\{q_n(x)\}_{n=0}^\infty$ with corresponding eigenvalues $\{\gamma_n\}_{n=0}^\infty$, so that
 \[
 T[q_n(x)]=\sum\limits_{k=0}^\infty Q_k(x)D^k[q_n(x)]=\gamma_nq_n(x) \qquad (n\in\mathbb{N}_0).
 \]
Then the polynomial coefficients of $T$ are given by
\begin{equation}\label{eq:affineqk}
Q_k(x)=\frac{(-1)^k(\alpha x + \beta)^k}{k!~a^k}\bigg(\gamma_0-\sum\limits_{j=1}^k\dbinom{k}{j}(-1)^{j+1}\gamma_j\bigg), \qquad k\in\mathbb{N}.
\end{equation}
Consequently, $T$ is monotone.
\end{theorem}

\begin{proof}
 If $n=0,$ a direct calculation verifies that
\begin{equation}\label{1.2}
Q_0(x)=\gamma_0.
\end{equation}
Now let $n\geq 1.$ Then,  
\begin{align*}
\sum\limits_{k=0}^n Q_k(x)D^k[q_n(x)]&=\sum\limits_{k=0}^nQ_k(x)\left(\alpha^k \frac{n!}{(n-k)!}c_n(\alpha x+\beta)^{n-k}\right)\\
&=c_n\sum_{k=0}^nQ_k(x)\left(\frac{\alpha^k}{c_{n-k}}\frac{n!}{(n-k)!}q_{n-k}(x)\right)\\
&=c_n\sum_{k=0}^n\bigg[\frac{(-1)^kq_k(x)}{k!c_k\alpha^k}(\gamma_0-\sum_{j=1}^k\dbinom{k}{j}(-1)^{j+1}\gamma_j)\bigg]\bigg(\frac{\alpha^k}{c_{n-k}}\cdot\frac{n!}{(n-k)!}q_{n-k}(x)\bigg).\\
\end{align*}
Note that 
\[
q_k(x) q_{n-k}(x)=c_kc_{n-k}(\alpha x+\beta)^n=\frac{c_kc_{n-k}}{c_n}q_n(x) \qquad (n \in \mathbb{N}, 0\leq k \leq n).
\]
Thus,

\begin{eqnarray*}
\sum\limits_{k=0}^n Q_k(x)D^k[q_n(x)]&=&q_n(x)\bigg[\sum_{k=0}^n(-1)^k\dbinom{n}{k}\bigg(\gamma_0-\sum_{j=1}^k(-1)^{j+1}\gamma_j\bigg)\bigg]\\
&=&q_n(x)\bigg[\sum_{k=0}^n(-1)^k\dbinom{n}{k}\gamma_0-\sum_{k=0}^n\dbinom{n}{k}(-1)^k\sum_{j=1}^k\dbinom{k}{j}(-1)^{j+1}\gamma_j)\bigg]\\
&=&q_n(x)\bigg[-\sum\limits_{k=0}^n\dbinom{n}{k}(-1)^k\sum\limits_{j=0}^k\dbinom{k}{j}(-1)^{j+1}\gamma_j\bigg]\\
&=&q_n(x)\bigg[-\sum_{k=0}^n\sum\limits_{j=0}^k\dbinom{n}{k}\dbinom{k}{j}(-1)^k(-1)^{j+1}\gamma_j\bigg]\\
&=&q_n(x)\bigg[\sum_{k=0}^n\sum_{j=0}^k\dbinom{n}{k}\dbinom{k}{j}(-1)^{k-j}\gamma_j\bigg]\\
&=&\gamma_n q_n(x),
\end{eqnarray*}
where the last equality is a consequence of a straightforward combinatorial identity (see for example \cite[p.\,49]{riordan}). Finally, the fact that $T$ is monotone follows from the representation in (\ref{eq:affineqk}), since for each $k$, the polynomial $Q_k(x)$ is either identically zero, or of degree $k$. 
\end{proof}
The reader may wonder whether an operator being diagonal always implies its monotonicity, regardless of the basis in question. This turns out not to be the case. If an operator is diagonal with respect to a basis other than (an affine transformation of) the standard one, diagonality alone is not sufficient to ensure monotonicity.
We provide here two examples of linear operators, in order to demonstrate that diagonality (even with the additional requirement of hyperbolicity preservation) need not imply monotonicity.
\begin{example}\label{ex:notmonotone} Consider Hermite's differential equation $(D^2-2xD+\lambda)[y]=0$. The Hermite polynomials $\mathcal{H}_n$ are solutions to the equation with $\lambda=n$, and consequently the operator $T=2xD-D^2$ is Hermite diagonal. It is also hyperbolicity preserving, but is clearly not monotone. \\
The Legendre-diagonal operator $T$ with eigenvalues $\seq{k^2+\alpha k+\beta}$, $\alpha \neq 1$ is given by (see for example \cite[\S 5.2]{tominvolve})
\begin{eqnarray*}
T&=&\beta+(1+\alpha)xD-\left(\frac{2+\alpha-3x^2}{3}\right)D^2+\frac{2}{15}(\alpha-1)xD^3-\frac{(\alpha-1)(1+4x^2)}{105}D^4\\
&+&(\alpha-1)\sum_{k=5}^{\infty} T_k(x)D^k.
\end{eqnarray*}
Again we see that the operator is not monotone. 
\end{example}
In connection with Example \ref{ex:notmonotone} we direct the reader's attention to the following recent result of R. Bates:
\begin{corollary}(\cite[p.\,39]{bates}) \label{cor:bates} Let $T$ be a hyperbolicity preserving diagonal differential operator,
\[
T[B_n(x)]:=\left(\sum_{k=0}^{\infty} Q_k(x)D^k \right)B_n(x)=\gamma_n B_n(x), \qquad n   \in \mathbb{N}_0,
\]
where $0<\gamma_k \leq \gamma_{k+1}$ for every $k \in \mathbb{N}_0$. If $\seq{\gamma_k}$ can be interpolated by a polynomial of degree $n$, then $\deg Q_k(x)=k$ for $0 \leq k \leq n$. If $\seq{\gamma_k}$ cannot be interpolated by a polynomial, then $\deg Q_k(x)=k$ for all $k \in \mathbb{N}_0$.
\end{corollary}
Our first example illustrates that Bates' result is best possible, in that the eigenvalues of $T=1+2xD-D^2$ are interpolated by a polynomial of degree one, namely $p(x)=x+1$, but $\deg (-1)=0 \neq 2$.\\
In light of Example \ref{ex:notmonotone}, if we wish to ensure the monotonicity of a diagonal hyperbolicity preserver $T$, it must be an infinite order differential operator. Our final result deals with such a family of operators. In order to be able to state the result, we make the following definition.
\begin{definition}  The generalized Hermite polynomials with parameter $\alpha>0$ are defined as
\[
\mathcal{H}_n^{(\alpha)}(x) := (-\alpha)^n \exp(x^2/2\alpha) D^n \exp(-x^2/2\alpha) \qquad (n=0,1,2,\dots).
\]
\end{definition}
\begin{theorem} \label{thm:hermitediag} Let $\alpha >0$, and suppose that $T:\mathbb{R}[x] \to \mathbb{R}[x]$ is a linear operator that is diagonal with respect to a generalized Hermite basis $\left\{ \mathcal{H}^{(\alpha)}_n \right\}_{n=0}^{\infty}$, and is an infinite order differential operator. If $T$ is a hyperbolicity preserver, then $T$ is monotone. 
\end{theorem}
\begin{proof} By Theorem 3 in \cite[p.\,465]{tomandrzej1}, the coefficient polynomials in the differential operator representation of $T$ are given by
\begin{equation} \label{eq:hermiteqk}
Q_k(x)=\sum_{j=0}^{\lfloor k/2 \rfloor} \frac{(-\alpha)^j}{j!(k-2j)!}g_{k-j}^*(-1) \mathcal{H}^{(\alpha)}_{k-2j}(x), \qquad (k=0,1,2,\ldots).
\end{equation} 
Assume without loss of generality, that the eigenvalues of $T$ are all positive. Since $T$ is a hyperbolicity preserver, its eigenvalues form an $\mathcal{H}^{(\alpha)}$-multiplier sequence, and hence are non-decreasing. It follows that the sequence $\seq{g_k^*(-1)}$ is a classical multiplier sequence (see for example \cite[Lemma 3, p.\,468]{tomandrzej1}). Consequently, if $g_n^*(-1)=0$ for some $n$, then $g_m^*(-1)=0$ for all $m \geq n$. Equation (\ref{eq:hermiteqk}) would then imply that $Q_k(x) \equiv 0$ for all $k\geq 2n$, and hence $T$ would be a finite order operator contrary to the assumptions of the theorem. We conclude that $g_k^*(-1) \neq 0$ for any $k$, and whence $\deg Q_k(x)=k$ for all $k \geq 0$. The proof is complete.
\end{proof}
\section{Open problems}
While Section \ref{s:2} provides an example of an infinite order hyperbolicity preserver which is not monotone, we wonder whether adding the assumption that $\dim \text{Range}(T) \geq 3$ would indeed be sufficient to ensure monotonicity of a linear operator:
\begin{problem} \label{prob:T*} Suppose that $\displaystyle{T=\sum_{k=0}^{\infty} Q_k(x)D^k}$ is an infinite order differential operator, such that $\dim \text{Range}(T) \geq 3$. Determine whether $T$ has to be monotone.
\end{problem}
We note that a possible approach to this problem could be based on the following reasoning. Given an operator $T=\sum Q_k(x) D^k$, set $T^*:=\sum Q_k^*(x)D^k$, where $Q_k^*(x)$ denotes the reverse of the polynomial $Q_k(x)$. If
\[
G_{T^*}(z,w)=\sum_{k=0}^{\infty} \frac{(-1)^k T^*[z^k]w^k}{k!} \in \overline{\mathcal{H}_2}(\mathbb{R}),
\]
then acting on $G_{T^*}(z,w)$ by the non-negative multiplier sequence $\{1,0,0,0, \ldots\}$ in the $z$ variable produces an element of the Laguerre-P\'olya class (and hence a real entire function with only real zeros). On the other hand,
\[
\{1,0,0,\ldots\}[G_{T^*}(z,w)]=\sum_{k=0}^{\infty} \frac{Q_k^*(0)}{k!}w^k.
\] 
Since the numbers $Q_k^*(0)$ are exactly the leading coefficients of the polynomials $Q_k(x)$, there is a connection between $T^*$ being a hyperbolicity preserver, and the degrees of the polynomials $Q_k(x)$. In order for this approach to produce and answer to Problem \ref{prob:T*}, one would have to show that if $T$ is reality preserving, so it $T^*$, and then understand the aforementioned connection precisely.\\
\indent The use of a generalized Hermite basis in Theorem \ref{thm:hermitediag} prompted the following problem concerning monotonicity of a diagonal hyperbolicity preserver.
\begin{problem}\label{prob:diffeq} Suppose that $\seq{B_k(x)}$ is a basis for $\mathbb{R}[x]$ such that $\deg B_k(x)=k$ for $k=0,1,2,\ldots$ and such that each $B_k(x)$ is solution to a finite (fixed) order differential equation 
\[
a_0(k)+\sum_{j=1}^M a_j(x) y^{(j)}=0, 
\]
where $a_0(k)$ is a constant depending on $k$. Determine whether or not an infinite order differential hyperbolicity preserving operator $T$ that is diagonal with respect to the basis $\seq{B_k(x)}$ must be monotone. 
\end{problem}
Though more general, Problem \ref{prob:diffeq} would settle the question whether operators, that are diagonal with respect to the classical orthogonal bases (and satisfy the additional hypotheses of Theorem \ref{thm:hermitediag}) are monotone. What made proving Theorem \ref{thm:hermitediag} possible was that we had an explicit expression of the coefficient polynomials $Q_k(x)$. To the best of our knowledge, at this time no explicit expressions are known for the coefficient polynomials of operators that are diagonal with respect to the other classical orthogonal bases. 
\\

\bigskip

\noindent ${}^{\star}$  Department of Mathematics\\
5245 N. Backer Ave, M/S PB 108\\
California State University, Fresno 93740-8001

\bigskip

\noindent ${}^{\dag}$  Department of Mathematics and Computer Science\\
Muskingum University\\
163 Stormont St.\\
New Concord, OH 43762

\bigskip

\noindent ${}^{\ddag}$  Department of Mathematics and Statistics\\
University of Wisconsin - La Crosse\\
1725 State St.\\
La Crosse, WI 54601

\end{document}